\documentclass[a4paper,12pt,reqno]{amsart}

\usepackage[utf8]{inputenc}
\usepackage{latexsym}
\usepackage{amssymb}
\usepackage{amsmath}
\usepackage{enumerate}
\usepackage{xcolor}
\usepackage{url}
\usepackage{mathrsfs}
\usepackage{ulem}
\newtheorem{theorem}{Theorem}[section]
\newtheorem{lemma}[theorem]{Lemma}
\newtheorem{proposition}[theorem]{Proposition}
\newtheorem{corollary}[theorem]{Corollary}
\newtheorem{example}[theorem]{Example}

\newtheorem{definition}[theorem]{Definition}

\def\N{\mathbb N}
\def\R{\mathbb R}
\def\C{\mathbb C}
\def\Z{\mathbb Z}

\title[]{
Spectral decomposition of power-bounded operators: The finite spectrum case
}

\author{
Shiho Oi
}
\address{
Department of Mathematics, Faculty of Science, 
Niigata University, Niigata 950-2181, Japan.
}
\email{shiho-oi@math.sc.niigata-u.ac.jp
}

\author{
Jyamira~Oppekepenguin
}
\address{
IKAKEN Niigata
}
\email{oppekepenguin@gmail.com
}


\begin{document}

%\maketitle\textbf{}

\begin{abstract}
In this paper, we investigate power-bounded operators, including surjective isometries, on Banach spaces. Koehler and Rosenthal asserted that an isolated point in the spectrum of a surjective isometry on a Banach space lies in the point spectrum, with the corresponding eigenspace having an invariant complement. However, they did not provide a detailed proof of this claim, at least as understood by the authors of this manuscript. Here, 
%applying 
by applications of a theorem of Gelfand and the Riesz projections, we demonstrate that the theorem of Koehler and Rosenthal holds for any power-bounded operator on a Banach space. This not only furnishes a detailed proof of the theorem but also slightly generalizes its scope.
As a result, we establish that if $T: X \to X$ is a power-bounded operator on a Banach space $X$ whose spectrum consists of finitely many points ${\lambda_1, \lambda_2, \dots, \lambda_m}$, then for every $1 \leq i, j \leq m$, there exist projections $P_j$ on $X$ such that 
$P_iP_j=\delta_{ij}P_i$, $\sum_{j=1}^mP_j=I$, and $T=\Sigma_{j=1}^m \lambda_j P_j$. 
It follows that such an operator $T$ is an algebraic operator.
\end{abstract}

\maketitle
%%%%%%%%%%%%%%%%%%%%%%%%%%%%%%%%%%%%%%%%%%%%%
%%%%%%%%%%%%%%%%%%%%%%%%%%%%%%%%%%%%%%%%%%%%%
%%%%%%%%%%%%%%%%%%%%%%%%%%%%%%%%%%%%%%%%%%%%%
\section{Introduction}
%%%%%%%%%%%%%%%%%%%%%%%%%%%%%%%%%%%%%%%%%%%%%
%%%%%%%%%%%%%%%%%%%%%%%%%%%%%%%%%%%%%%%%%%%%%
Sz.-Nagy \cite{SzN} characterized a power-bounded operator on a  Hilbert space: an invertible bounded complex linear operator on a Hilbert space is power-bounded if and only if it is similar to a unitary operator. Consequently, a power-bounded operator on a Hilbert space admits a spectral decomposition by applying a spectral measure. What about the case of a Banach space operator?
The origin of the problem can be traced back to a theorem of Gelfand \cite{gelfand}, which characterizes the identity operator as a power-bounded operator with the spectrum $\{1\}$. Zem\'anek \cite{zemanek} 
provides a very informative and well-written account of developments related to this theorem. 
Koehler and Rosenthal \cite{KR} investigated isometries with an isolated spectrum. 
Botelho and Ili\v sevi\'c \cite{BI} studied isometries with finite spectrum. See also \cite{ILP,ILW}.
Berkson, Gillespie, and Muhly explored 
spectral decompositions of Banach space operators in a systematic manner in \cite{BGM}.

This paper studies a power-bounded operator on an arbitrary complex  Banach space with an isolated spectrum. 
In this paper, a power-bounded operator  $T$ on a Banach space is 
%Recall that 
an invertible bounded operator which satisfies that 
\begin{equation}\label{spt}
    \sup_{n\in \Z}\|T^n\|<\infty,
    \end{equation}
where $\Z$ is the set of all integers. 
A surjective isometry is a power-bounded operator. 
A theorem of Koehler and Rosenthal \cite{KR} states that if a surjective isometry on a Banach space has an isolated point in the spectrum, then the isolated point is a point spectrum (eigenvalue) of which eigenspace has a complemented subspace. The proof in \cite{KR} asserts that "This follows from Stampfli's work",  
which refers to the reference \cite{stampfli}. 
However, while in \cite{stampfli} Stampfli primarily studied adjoint-abelian operators, which are generalizations of self-adjoint operators, there appears to be no exact statement or proof, except for Lemma 6 and its proof, that directly addresses iso-abelian operators.
It is worth noting that Koehler and Rosenthal demonstrated that an iso-abelian operator is equivalent to a surjective isometry \cite[Corollary 1]{KR}. Nevertheless, the authors of this paper could not find a corresponding proof of the aforementioned theorem of Koehler and Rosenthal in the existing literature. 
One of the primary objectives of this paper is to provide an exact proof of a slight generalization of this theorem,  
%the critical theorem of Koehler and Rosenthal, 
stated as Theorem \ref{KR}, which we derive by applying a theorem of Gelfand \cite{gelfand} and the Riesz projections. 
As a corollary, we exhibit the spectrum decomposition of a power-bounded operator with a finite spectrum in Corollary \ref{sd}. 
Then, by applying a theorem of Ili\v sevi\'c \cite{ilisevic} (Proposition \ref{ilis}), we see that such an operator is algebraic.

Throughout the paper, $X$ is a complex Banach space unless otherwise noted. The algebra of all bounded complex linear operators on $X$ is denoted by $B(X)$. The identity operator on $X$ is $I$. The spectrum of $T\in B(X)$ is $\sigma(T)$. A point spectrum (eigenvalue)  is an element $\lambda\in \sigma(T)$ such that there exists $0\ne x\in X$ with $T(x)=\lambda x$. If $\sigma(T)$ consists only of $0$, we say that $T$ is a quasi nilpotent.

%%%%%%%%%%%%%%%%%%%%%%%%%%%%%%%%%%%%%%%%%%%%%
%%%%%%%%%%%%%%%%%%%%%%%%%%%%%%%%%%%%%%%%%%%%%
%%%%%%%%%%%%%%%%%%%%%%%%%%%%%%%%%%%%%%%%%%%%%
\section{A theorem of Gelfand}
%%%%%%%%%%%%%%%%%%%%%%%%%%%%%%%%%%%%%%%%%%%%%
%%%%%%%%%%%%%%%%%%%%%%%%%%%%%%%%%%%%%%%%%%%%%
%%%%%%%%%%%%%%%%%%%%%%%%%%%%%%%%%%%%%%%%%%%%%
For the readers' convenience and for the completeness of the paper, we exhibit a proof of a theorem of Gelfand (Theorem \ref{gelfand} \cite[Satz 1]{gelfand}). 
Although Satz 1 in \cite{gelfand} seems to be stated for commutative Banach algebras, but it holds for non-commutative Banach algebras
%$B(X)$
as well.
To prove Theorem \ref{gelfand}, we apply the following version of the Liouville theorem, as in the same manner as in \cite{gelfand}.
\begin{lemma}\label{liouville}
Let $\chi\colon \C \to X$ be an entire map. 
Let $P(\theta)=\sum_{l=k}^mc_le^{il\theta}$ be a triangle polynomial with $c_k\ne 0$. Suppose that there exist real numbers $M>0$, $r_0>0$ and a non-negative integer $\alpha$ such that
\begin{equation}\label{res}
|P(\theta)|\|\chi(re^{i\theta})\|\le Mr^{\alpha}, \quad \forall r\ge r_0, \,\,\theta\in \R.
\end{equation}
Then $\chi(w)$ is a polynomial of $w$ with a degree at most $\alpha$. In particular, if $\alpha=0$, then $\chi(w)$ is a constant map.
\end{lemma}
\begin{proof}
We have a Taylar expansion
\[
\chi (w)=\sum_{s=0}^\infty a_sw^s.
\]
For $r>0$ and $p\in \N$, put
\[
J(r,p)=r^{-p}\frac{1}{2\pi}\int_0^{2\pi}P(\theta)\chi(re^{i\theta})e^{-(k+p)i\theta}d\theta.
\]
Suppose that $p\in \N$ is an arbitrary positive integer with $p>\alpha$. Letting $r\to \infty$ we have
\begin{equation}\label{1.6}
\|J(r,p)\|\le Mr^{-(p-\alpha)}\to 0
\end{equation}
by the inequality \eqref{res}. We have
\[
\frac{1}{2\pi}\int_0^{2\pi}e^{i(t+s-(k+p))\theta}
=
\begin{cases}
    1, & t+s=k+p \\
    0, &\text{otherwise},
\end{cases}
\]
hence we get
\begin{multline}
        J(r,p)=r^{-p}\frac{1}{2\pi}\int^{2\pi}_0\sum_{t=k}^mc_{t}e^{it\theta}\sum_{s=0}^\infty a_s r^s e^{is\theta}e^{-(k+p)i\theta} d\theta
       \\ =
        \sum_{t=k}^m\sum_{s=0}^\infty c_t a_s r^{s-p}
        \times 
        \begin{cases}
            1,& t+s=k+p
            \\
            0,& \text{otherwise}
        \end{cases}
        \\
        =\sum_{t=k}^mc_t a_{p+k-t}r^{k-t},
            \end{multline}
where $a_{p+k-t}=0$ for $p+k-t<0$. Since $k\le t\le m$ we have
\[
\sum_{t=k}^mc_{t}a_{p+k-t}r^{k-t}\to c_ka_p
\]
as $r\to \infty$. Since $c_k\ne0$ we have $a_p=0$ by \eqref{1.6}. This holds for every $p$ with $p>\alpha$. 
\end{proof}
\begin{lemma}\label{2.2}
    %Let $X$ be a complex Banach space and 
    Let $a$ be an invertible element in a Banach algebra $B$ with unit $e$.
   % $T\in B(X)$ be an invertible operator. 
   Suppose that %$a$ satisfies that 
    \begin{equation*}\label{spt}
    \sup_{n\in \Z}\|a^n\|<\infty.
    \end{equation*}
    Then, every element in the spectrum of $a$ is unimodular. 
\end{lemma}
\begin{proof}
    Suppose that $\mu\in \C$ with $0<|\mu|<1$. Then $\sum_{n=0}^\infty(\mu a)^n$ converges and 
    \begin{equation}\label{2.2.1}
    \sum_{n=0}^\infty(\mu a)^n=(e-\mu a)^{-1}
    \end{equation}
     by the condition \eqref{spt}. Hence $\frac{1}{\mu}\not\in \sigma(a)$. 

     Suppose that $\mu \in \C$ with $1<|\mu|$. Again by \eqref{spt}, we have
     \begin{equation}\label{2.2.2}
     \sum_{n=1}^\infty(\mu a)^{-n}=-(e-\mu a)^{-1}.
     \end{equation}
     Thus $\frac{1}{\mu}\not\in \sigma(a)$. 

     As $a$ is invertible, $0\not\in \sigma(a)$.

     It follows from the above we see that $\sigma(a)\subset \{z\mid |z|=1\}$.
\end{proof}
\begin{theorem}[Satz 1 in \cite{gelfand}]\label{gelfand}
    %Let $X$ be a complex Banach space and 
    Let $a$ be an invertible element in a Banach algebra $B$ with unit $e$.
   % $T\in B(X)$ be an invertible operator. 
   Suppose that %$a$ satisfies that 
    \begin{equation}\label{spt1}
    \sup_{n\in \Z}\|a^n\|<\infty.
    \end{equation} 
    Suppose further that 
    the spectrum $\sigma(a)=\{\lambda\}$ is a singleton. Then 
    %$T$ is a surjective isometry, particularly 
    $a=\lambda e$.
\end{theorem}
\begin{proof}
    By Lemma \ref{2.2} we have $|\lambda|=1$. Put $N=\bar{\lambda}a-e$. Then $\sigma(N)=\{0\}$, that is, $N$ is a quasi nilpotent,  as $\sigma(a)=\{\lambda\}$ and $|\lambda|=1$. We prove $N=0$. To prove it, let
    \[
    \chi(w)=\left(e+\left(w+\frac12\right)N\right)^{-1}, \quad w\in \C,
    \]
    which is well defined since $e+zN$ is invertible for every $z\in \C$ as $\sigma(e+zN)=\{1\}$. Let $w\in \C$ be as $|w|>1$. Then $w=re^{i\theta}$, where $\theta$ is the argument of $w$ and $r=|w|$. Put $P(\theta)=e^{-i\theta}+e^{i\theta}$. Letting $M=\sup_{n\in \Z}\|a^n\|$, we show
    \[
    \|P(\theta)\chi(re^{i\theta})\|\le 3M
    \]
    for every $r>1$ and real number $\theta$.
    It will follow by Lemma \ref{liouville} that  $\chi$ is a constant map, then $N=0$ as desired. In the following, $w=re^{i\theta}$ for $r>1$ and a real number $\theta$. Put 
    \[
    \mu=\frac{w+\frac12}{w-\frac12}.
    \]
    By calculation, we have
\begin{multline}\label{wl}
    -\left(w-\frac12\right)\left(e+\left(w+\frac12\right)N\right)^{-1}(e-\mu \bar{\lambda}a)
    \\
    =
    -\left(w-\frac12\right)\left(e+\left(w+\frac12\right)N\right)^{-1}
    \\
    +\left(w-\frac12\right)\left(e+\left(w+\frac12\right)N\right)^{-1}\mu \bar{\lambda}a
    \\
    =    \left(e+\left(w+\frac12\right)N\right)^{-1}+\left(e+\left(w+\frac12\right)N\right)^{-1}\left(\left(w+\frac12\right)N\right)
    \\
    =    \left(e+\left(w+\frac12\right)N\right)^{-1}\left(e+\left(w+\frac12\right)N\right)=e.
\end{multline}
If $\theta\ne \frac{\pi}{2}+m\pi$ for any integer $m$, then we have $|\mu\bar{\lambda}|\ne 1$. Applying similar calculation as equations \eqref{2.2.1} and \eqref{2.2.2} we have 
\[
\|(e-\mu \bar{\lambda}a)^{-1}\|\le \frac{M}{||\mu|-1|}.
\]
Hence by \eqref{wl} we have
\begin{equation}\label{12345}
\begin{split}
\|\chi(re^{i\theta})\|&=\left\|\left(e+\left(w+\frac12\right)N\right)^{-1}\right\|\\
    &=\left\|\frac{(e-\mu \bar{\lambda}a)^{-1}}{w-\frac12}\right\|
    \le
    \frac{M\left(\left|w+\frac12\right|+\left|w-\frac12\right|\right)}{\left|\left|w+\frac12\right|^2-\left|w-\frac12\right|^2\right|}.
    \end{split}
\end{equation}
As $r>1$ we have
\[
\left|\left|w+\frac12\right|^2-\left|w-\frac12\right|^2\right|
=2r|\cos \theta|,\quad 
\left|w+\frac12\right|+\left|w-\frac12\right|\le 3r.
\]
It follows by 
\eqref{12345} we have
\begin{equation}\label{hutous}
\|P(\theta)\chi(re^{i\theta})\|=|P(\theta)|\|\chi(re^{i\theta})\|
\le
3M
\end{equation}
if $\theta\ne \frac{\pi}{2}+m\pi$ for any integer $m$. 
On the other hand, for any $r>1$, the function $P(\theta)\chi(re^{i\theta})$ is continuous with respect to $\theta$. Thus \eqref{hutous} also holds for $\theta= \frac{\pi}{2}+m\pi$ for an integer $m$. 
It follows by Lemma \ref{liouville} that $\chi(w)=(e+(w+\frac12)N)^{-1}$ is a constant map as is desired.

\end{proof}
\begin{example}
    Let $T\in B(X)$ be such that $T=VSV^{-1}$, where $S\in B(X)$ is a surjective isometry and $V\in B(X)$ is an invertible operator. 
    Then $T$ is power-bounded. A celebrated theorem of Sz.-Nagy \cite{SzN} states that the converse holds if $X$ is a Hilbert space. 
\end{example}

%%%%%%%%%%%%%%%%%%%%%%%%%%%%%%%%%%%%%%%%%%%%%
%%%%%%%%%%%%%%%%%%%%%%%%%%%%%%%%%%%%%%%%%%%%%
%%%%%%%%%%%%%%%%%%%%%%%%%%%%%%%%%%%%%%%%%%%%%
\section{A proof of a theorem of Koehler and Rosenthal}
%%%%%%%%%%%%%%%%%%%%%%%%%%%%%%%%%%%%%%%%%%%%%
%%%%%%%%%%%%%%%%%%%%%%%%%%%%%%%%%%%%%%%%%%%%%
%%%%%%%%%%%%%%%%%%%%%%%%%%%%%%%%%%%%%%%%%%%%%
We say that $P\in B(X)$ is a projection if $P^2=P$. We do not assume a further property, such as the self-adjointness, even if $X$ is a Hilbert space. We recall the definition of the Riesz projection.
\begin{definition}
    Let $T\in B(X)$. Suppose that $\sigma(T)$ is the disjoint union of two %non-empty 
    closed subsets $\sigma$ and $\tau$. 
    Suppose that $\Gamma$ is a Cauchy contour in the resolvent set of $T$ around $\sigma$ separating $\sigma$ from $\tau$. The operator
    \[
    P_\sigma=\frac{1}{2\pi i}\int_\Gamma(T-wI)^{-1}dw
    \]
    is called the Riesz projection of $T$ corresponding to $\sigma$.
\end{definition}
Note that the Riesz projection does not depend on the choice of the Cauchy contour around $\sigma$ separating $\sigma$ from $\tau$. 
The Riesz projection is, in fact, a projection \cite[Lemma 2.1]{GGK}. 
Please refer to \cite{GGK} for further properties of the Riesz projection.

The following is a generalization of Theorem 5 in \cite{KR}, which is stated for isometries. 
\begin{theorem}\label{KR}
    Let $X$ be a complex Banach space and $T\in B(X)$ an invertible operator. Suppose that $T$ satisfies that 
    \begin{equation}\label{spt2}
    \sup_{n\in \Z}\|T^n\|<\infty.
    \end{equation}
    Suppose that $\lambda\in \sigma(T)$ is an isolated point in $\sigma(T)$. 
    Then $\lambda$ is a point spectrum of $T$, and the eigenspace of $T$ corresponding to $\lambda$ is $P_\lambda(X)$ and it has an invariant complement $\operatorname{ker}P_\lambda$, where $P_\lambda$ is the Riesz projection of $T$ corresponding to $\lambda$. 
\end{theorem}
\begin{proof}
    By Lemma \ref{2.2} we have $|\lambda|=1$. 
    Let $P_\lambda$ be the Riesz projection 
    \[
    P_\lambda=\frac{1}{2\pi i}\int_\Gamma (T-w I)^{-1}dw,
    \]
    where $\Gamma$ is a Cauchy
contour (in the resolvent set of $T$) around $\lambda$ separating $\lambda$ from $\sigma(T)\setminus\{\lambda\}$ (see \cite{GGK}). 
Applying Theorem 2.2 in \cite{GGK}, we have
$M_\lambda\oplus L_\lambda=X$, where $M_\lambda=P_\lambda(X)$ and $L_\lambda=\operatorname{ker} P_\lambda$, $M_\lambda$ and $L_\lambda$ are $T$ invariant subspaces, and $\sigma(T|M_\lambda)=\{\lambda\}$ and $\sigma(T|L_\lambda)=\sigma(T)\setminus\{\lambda\}$. Note that $M_\lambda\ne \{0\}$. (Suppose that $M_\lambda=\{0\}$. Then $L_\lambda=X$, hence $\sigma(T|L_\lambda)=\sigma(T)$, which is a contradiction since $\sigma(T|L_\lambda)=\sigma(T)\setminus\{\lambda\}$.)  We have $T|M_\lambda\in B(M_\lambda)$ such that $\sigma(T|M_\lambda)=\{\lambda\}$. As $\|(T|M_\lambda)^n\|\le \|T^n\|$ for every integer $n$, we have $\sup_{n\in \Z}\|(T|M_\lambda)^n\|<\infty$. Applying Theorem \ref{gelfand} for $T|M_\lambda$ we have $T|M_\lambda=\lambda I_{M_\lambda}$, so that $T(x)=\lambda x$ for every $x\in M_\lambda$.  
Note that $M_\lambda$ is the eigenspace of $T$ for $\lambda$. 
The reason is as follows. 
As $T|M_\lambda =\lambda I_{M_\lambda}$, $M_\lambda$ is a subspace of the eigenspace for $\lambda$. We prove the converse. 
Suppose that $0\ne y\in X$ satisfies $T(y)=\lambda y$: $y$ is an eigenvector for $\lambda$. As $X=M_\lambda\oplus L_\lambda$, $y=y_M+y_L$ with $y_M\in M_\lambda$ and $y_L\in L_\lambda$. Then 
\[
\lambda y_M+\lambda y_L=\lambda y=T(y)=T(y_M)+T(y_L)=\lambda y_M+T(y_L),
\]
so
\[
\lambda y_L=T(y_L)=T|L_\lambda(y_L).
\]
Suppose that $y_L\ne 0$. Then $\lambda\in \sigma(T|L_\lambda)$, which is a contradiction since $\sigma(T|L_\lambda)=\sigma(T)\setminus\{\lambda\}$. It follows that $y_L=0$, hence $y=y_M\in M_\lambda$.
\end{proof}
The following is a generalization of a theorem of Gelfand \cite{gelfand}.
\begin{corollary}\label{sd}
    Suppose that $T\in B(X)$ is an invertible operator such that the spectrum $\sigma(T)$ is a finite set $\{\lambda_1, \lambda_2, \dots, \lambda_m\}$. 
    Then $T$ satisfies the condition  
    \begin{equation}\label{spt2}
    \sup_{n\in \Z}\|T^n\|<\infty
    \end{equation}
    if and only if $|\lambda_j|=1$ and there exists a projection $P_j$ for every $1\le j\le m$ such that 
    $T=\sum_{j=1}^m\lambda_jP_j$, which satisfies
    the following {\rm{(i)}} and {\rm{(ii)}}. 
    \begin{itemize}
        \item[(i)] The image $P_j(X)$ is the eigenspace of $T$ corresponding to $\lambda_j$ for every $1\le j\le m$. 
        \item[(ii)] $P_iP_j=0$ for every pair $(i,j)$ with $i\ne j$ and $I=\sum_{j=1}^mP_j$.
        
    \end{itemize}
    %In this case, every eigenspace $P_j(X)$ and $\operatorname{ker}P_j$ are $T$-invariant subspaces such that $P_j(X)\oplus \operatorname{ker}P_j=X$,  
    In this case, 
    %$T=\sum_{j=1}^m\lambda_jP_j$ and 
    the projection $P_j$ is unique in the sense that it is the Riesz projection of $T$ corresponding to the spectrum $\lambda_j$ for every $1\le j\le m$. 
    Hence, every eigenspace $P_j(X)$ and $\operatorname{ker}P_j$ are $T$-invariant subspaces such that $P_j(X)\oplus \operatorname{ker}P_j=X$.
    
\end{corollary}

\begin{proof}
Suppose that $|\lambda_j|=1$ and $T=\sum_{j=1}^m\lambda_jP_j$ for a projection $P_j$ with the conditions (i) and  (ii).
As $P_j$ is a projection and $P_jP_i=0$ for every pair $(i,j)$ with $i\ne j$, we have $T^n=\sum_{j=1}^m{\lambda_j}^nP_j$ for every positive integer $n$. We also have $T^{-1}=\sum_{j=1}^m{\lambda_j}^{-1}P_j$, so $T^{-n}=\sum_{j=1}^m{\lambda_j}^{-n}P_j$ for every positive integer $n$. As $|\lambda_j|=1$  for every $j$, we have 
$\sup_{n\in \Z}\|T^n\|\le \sum_{j=1}^m\|P_j\|<\infty$, which assures \eqref{spt2}.

Suppose conversely that \eqref{spt2} holds. By Theorem \ref{KR} we have $|\lambda_j|=1$ for every $1\le j\le m$. 
Let $Q_j$ denote the Riesz projection of $T$ corresponding to $\lambda_j$ for $1\le j\le m$. Due to Theorem \ref{KR} $Q_j(X)$ is the eigenspace of $T$ corresponding to $\lambda_j$ for $1\le j\le m$. Thus $Q_j$'s satisfy the condition (i). 
We prove that $T=\sum_{j=1}^m\lambda_jQ_j$ and 
$Q_j$ for $1\le j\le m$ satisfies (ii). 
We have
\begin{equation}\label{sumQj}
\sum_{j=1}^m Q_j=I.
\end{equation}
In fact, letting $\Gamma_j$ be a Cauchy contour (in the resolvent set of $T$) around $\lambda_j$ separating $\lambda_j$ from $\sigma(T)\setminus \{\lambda_j\}$, and $\Gamma$ a Cauchy contour around $\sigma(T)$, we have
\[
\sum_{j=1}^mQ_j=\sum_{j=1}^m\frac{1}{2\pi i}\int_{\Gamma_j}(T-wI)^{-1}dw=\frac{1}{2\pi i}\int_{\Gamma}(T-wI)^{-1}dw=I.
\]
It follows that $+_{j=1}^mQ_j(X)=X$. 
%Due to Theorem \ref{KR}, 
As $Q_j(X)$ is the eigenspace for $\lambda_j$, we have $Q_i(X)\cap Q_j(X)=\{0\}$ for $i \neq j$.  This implies that $\oplus_{j=1}^mQ_j(X)=X$. 
For every pair $i$ and $j$ for any $1\le i,j\le m$ with $i\ne j$, we have $Q_iQ_j=0$. (The reason is as follows. Due to the definition of the Riesz projection, we have $Q_{\{\lambda_i,\lambda_j\}}=Q_i+Q_j$, where 
$Q_{\{\lambda_i,\lambda_j\}}$ is the Riesz projection of $T$ corresponding to $\{\lambda_i,\lambda_j\}$. Thus $Q_i+Q_j$ is a projection. Since 
\[
Q_i+Q_j=(Q_i+Q_j)^2=Q_i+Q_j+Q_iQ_j+Q_jQ_i,
\]
we have $Q_iQ_j+Q_jQ_i=0$. Multiplying $Q_i$ for the equality from the left, we get 
\[
Q_iQ_j+Q_iQ_jQ_i=0
\]
as $Q_i^2=Q_i$. Thus 
\[
Q_iQ_j(1+Q_i)=0.
\]
As $\sigma(Q_i)\subset \{0,1\}$, $1+Q_i$ is invertible, we have $Q_iQ_j=0$. ) Thus $Q_j$ ($1\le j\le m$) satisfies (ii).

Suppose that $P_j$ ($1\le j\le m$) is a projection such that $T=\sum_{j=1}^m\lambda_jP_j$ which satisfies the conditions (i) and (ii). We prove $P_j=Q_j$ for every $1\le j\le m$. 
First we point out that the eigenspace of $T$ for $\lambda_j$ is $P_j(X)=Q_j(X)$. For any $x\in X$, $Q_j(x)\in Q_j(X)=P_j(X)$, hence $P_jQ_j(x)=Q_j(x)$. Thus $P_iQ_j(x)=P_iP_jQ_j(x)=0$ for every $x\in X$, which ensures $P_iQ_j=0$ for every pair $(i,j)$ with $i\ne j$. 
Similarly, we have $P_iQ_i=Q_i$ for every $1\le i\le m$. 
%We also have $Q_iP_j=0$ for every pair $(i,j)$ with $i\ne j$. 
It follows that for every $1\le i\le m$ and $x\in X$ we have
\[
P_iT(x)=\sum_{j=1}^m\lambda_jP_iP_j(x)=\lambda_iP_i(x), 
\]
and 
\[
P_iT(x)=\sum_{j=1}^m\lambda_jP_iQ_j(x)=\lambda_iQ_i(x).
\]
As $\lambda_i\ne0$ we conclude that $P_i(x)=Q_i(x)$ for every $1\le i\le m$ and $x\in X$.

As $P_j=Q_j$ is the Riesz projection, Theorem \ref{KR} asserts that $P_j(X)$ and $\operatorname{ker}P_j$ are $T$ invariant and $P_j(X)\oplus \operatorname{ker}P_j=X$. 
\end{proof}
Botelho and Ili\v sevi\'c \cite{BI} studied isometries with finite spectrum. 
We note that the following is proved by Ili\v sevi\'c \cite[Proposition 2.4]{ilisevic}; see also \cite[Proposiiton 1.2]{ILP}.
\begin{proposition}\label{ilis}
    Suppose that $T,P_1,P_2,\dots,P_m\in B(X)$ and $\lambda_1,\lambda_2,\dots,\lambda_m$ are distinct complex numbers. 
    The following conditions {\rm{(i)}} and {\rm{(ii)}} are equivalent.
\begin{itemize}
    \item[(i)] $T=\sum_{j=1}^m\lambda_jP_j$; $P_1,P_2,\dots P_m$ are projections such that $I=\sum_{j=1}^mP_j$ and $P_iP_j=0$ for every pair $(i,j)$ with $i\ne j$.  
    \item[(ii)] $\prod_{j=1}^m(T-\lambda_jI)=0$ and 
    \[
    P_j=\frac{\prod_{i\ne j}(T-\lambda_iI)}{\prod_{i\ne j}(\lambda_j-\lambda_i)},\qquad j=1,\dots, m.
    \]
\end{itemize}
\end{proposition}
By Corollary \ref{sd} and Proposition \ref{ilis}, we have the following. 
\begin{corollary}
    Let $T\in B(X)$ be a power-bounded operator. Suppose that 
        the spectrum $\sigma(T)$ is a finite set $\{\lambda_1, \lambda_2, \dots, \lambda_m\}$. 
    Then $T$ is an algebraic operator such that 
    $\prod_{j=1}^m(T-\lambda_jI)=0$.
\end{corollary}
\begin{example}
Let $\|\cdot\|$ be any norm on $\C^2$.
    Suppose that $A=\begin{bmatrix}
5 & -2 \\
12 & -5 \\
\end{bmatrix}$. 
We have
\[
\begin{bmatrix}
5 & -2 \\
12 & -5 \\
\end{bmatrix}
=
\begin{bmatrix}
1 & 1 \\
2 & 3 \\
\end{bmatrix}
\begin{bmatrix}
1 & 0 \\
0 & -1 \\
\end{bmatrix}
\begin{bmatrix}
1 & 1 \\
2 & 3 \\
\end{bmatrix}^{-1}.
\]
Hence, the matrix $A$ is power-bounded as an operator on $(\C^2,\|\cdot\|)$. We have the following.
The point spectrum (eigenvalues) of $A$ are 
$\pm 1$. The eigenspace $X_1$ of $A$ corresponding to $1$ is $\left\{r\begin{pmatrix}1\\2\end{pmatrix}\mid r\in \C\right\}$, and the eigenspace $X_{-1}$ corresponding to $-1$ is 
$\left\{r\begin{pmatrix}1\\3\end{pmatrix}\mid r\in \C\right\}$. 
Then $P=\begin{bmatrix}
3 & -1 \\
6 & -2 \\
\end{bmatrix}$
is a projection such that $P\C^2=X_1$, and 
$Q=\begin{bmatrix}
-2 & 1 \\
-6 & 3 \\
\end{bmatrix}$ is a projection such that $Q\C^2=X_{-1}$, and $PQ=QP=0$.
We have 
\[
A=1P+(-1)Q.
\]
\end{example}
\begin{example}
 Suppose that $A=\begin{bmatrix}
1 & 1 \\
0 & 1 \\
\end{bmatrix}$. Then $A$ is not power-bounded as an operator on $\C^2$ since $A^n=\begin{bmatrix}
1 & n \\
0 & 1 \\
\end{bmatrix}$ and $\sup_{n\in \Z}\|A^n\|=\infty$. 
The point spectrum (eigenvalue) is $1$ and the corresponding eigenspace for $1$ is $\left\{r\begin{pmatrix}1\\0\end{pmatrix}\mid r\in \C\right\}$. Furthermore,  
there is no projection $P$ such that $A=1P$. 
\end{example}

\subsection*{Acknowledgments}
The first-named author was supported by JSPS KAKENHI Grant Numbers JP24K06754.

\end{document}